\documentclass[12pt,reqno]{amsart}
\pdfoutput=1
\usepackage[headings]{fullpage}
\usepackage{amssymb,epsfig,amsmath,amsbsy,amsmath,bbm}
\usepackage{graphicx}
\usepackage{texdraw}
\usepackage{url}
\usepackage[bookmarks=true,%
    colorlinks=true,%
    linkcolor=blue,%
    citecolor=blue,%
    filecolor=blue,%
    menucolor=blue,%
    urlcolor=blue,%
    breaklinks=true]{hyperref}
\usepackage{amscd}

\newtheorem{theorem}{Theorem}[section]
\newtheorem{proposition}{Proposition}[section]
\theoremstyle{definition}
\newtheorem{lemma}[proposition]{Lemma}
\newtheorem{definition}[proposition]{Definition}
\newtheorem{remark}[proposition]{Remark}
\newtheorem{corollary}[proposition]{Corollary}

\newcommand{\1}{{1\!\!1}}

\usepackage{xspace}
\usepackage{amssymb}
\usepackage{euscript}

\def\to{\mathchoice
{\longrightarrow}
{\rightarrow}
{\rightarrow}
{\rightarrow}}

\def\hookrightarrow{\mathchoice
{\DOTSB\lhook\joinrel\relbar\joinrel\rightarrow}
{\DOTSB\lhook\joinrel\rightarrow}
{\DOTSB\lhook\joinrel\rightarrow}
{\DOTSB\lhook\joinrel\rightarrow}}

\swapnumbers

\newcommand{\Z}{\mathbb{Z}}
\newcommand{\C}{\mathbb{C}}
\newcommand{\N}{\mathbb{N}}

\newcommand{\Q}{\mathbb{Q}}

\newcommand{\om}{\omega}
\newcommand{\eps}{\varepsilon}
\renewcommand{\o}{\otimes}
\newcommand{\ohat}{\Hat{\otimes}}
\newcommand{\half}{{\textstyle\frac12}}

\DeclareMathOperator{\Hom}{Hom}

\DeclareMathOperator{\Ind}{Ind}
\DeclareMathOperator{\Res}{Res}

\newcommand{\p}{\partial}

\newcommand{\<}{\langle}
\renewcommand{\>}{\rangle}

\newcommand{\bull}{\bullet}
\renewcommand{\SS}{\mathbb{S}}

\DeclareMathOperator{\ch}{ch}

\DeclareMathOperator{\rk}{rk}

\DeclareMathOperator{\Exp}{Exp}
\DeclareMathOperator{\Log}{Log}
\renewcommand{\v}{{\EuScript V}}
\newcommand{\w}{{\EuScript W}}
\newcommand{\e}{{\EuScript E}}

\renewcommand{\b}{{\EuScript B}}
\renewcommand{\c}{{\EuScript C}}
\newcommand{\m}{{\EuScript M}}
\newcommand{\n}{{\EuScript N}}

\newcommand{\CO}{\mathcal{O}}

\renewcommand{\]}{{]\!]}}

\newcommand{\DD}{\mathbb{D}}
\DeclareMathOperator{\GL}{GL}
\DeclareMathOperator{\gl}{\mathfrak{gl}}
\newcommand{\Sp}{\operatorname{Sp}}
\newcommand{\fsp}{\operatorname{\mathfrak{sp}}}
\DeclareMathOperator{\Wedge}{\Lambda}

\DeclareMathOperator{\Diff}{Diff}
\DeclareMathOperator{\Schur}{\mathsf{S}}
\newcommand{\CM}{\mathcal{M}}
\newcommand{\CC}{\mathcal{C}}
\newcommand{\Int}{\mathbb{L}}
\newcommand{\HH}{\mathsf{H}}
\newcommand{\TT}{\mathsf{T}}

\renewcommand{\AA}{\mathsf{A}}
\DeclareMathOperator{\cch}{c}
\DeclareMathOperator{\IH}{IH}
\DeclareMathOperator{\Tor}{T}
\DeclareMathOperator{\Flag}{Flag}
\DeclareMathOperator{\VVert}{Vert}
\DeclareMathOperator{\Edge}{Edge}
\DeclareMathOperator{\Leg}{Leg}
\DeclareMathOperator{\lloop}{loop}
\newcommand{\tomega}{\tilde{\omega}}
\newcommand{\cc}{\mathcal{C}}
\newcommand{\dd}{\mathcal{D}}
\newcommand{\K}{\mathcal{K}}
\newcommand{\T}{\mathcal{T}}
\renewcommand{\aa}{\mathcal{A}}
\def\eqdef{\overset{\text{def}}{=}}
\def\ft{\mathfrak{t}}
\def\fu{\mathfrak{u}}
\newcommand{\Ga}{\Gamma}
\newcommand{\La}{\Wedge}
\newcommand{\tp}{\mathrm{top}}
\def\di{\diamond}
\def\longto{\longrightarrow}

\begin{document}

\title{Graph complexes and the symplectic character of the Torelli group}

\author{Stavros Garoufalidis}
\address{School of Mathematics \\
         Georgia Institute of Technology \\
         Atlanta, GA 30332-0160, USA \newline 
         {\tt \url{http://www.math.gatech.edu/~stavros}}}
\email{stavros@math.gatech.edu}

\author{E. Getzler}
\address{Department of Mathematics, \\ Northwestern University, \\
Evanston, IL 60208-2730.}
\email{getzler@math.nwu.edu}

\thanks{
{\em Key words and phrases: mapping class group, Torelli group, symplectic
representations, graph homology, Mumford conjecture.}
}


\begin{abstract}
The mapping class group of a closed surface of genus $g$ is an extension
of the Torelli group by the symplectic group. This leads to two natural
problems: (a) compute (stably) the symplectic decomposition of the lower 
central series of the Torelli group and (b) compute (stably) the Poincar\'e 
polynomial of the cohomology of the mapping group with coefficients in a 
symplectic representation $V$. Using ideas from graph cohomology, we give
an effective computation of the symplectic decomposition of the quadratic
dual of the lower central series of the Torelli group, and assuming the
later is Kozsul, it provides a solution to the first problem. This, together
with Mumford's conjecture, proven by Madsen-Weiss, provides a solution to
the second problem. Finally, we present samples of computations, up to degree 
13.
\end{abstract}

\maketitle

{\footnotesize
\tableofcontents
}

\section{Introduction}
\label{sec.intro}

The mapping class group $\Gamma_g$ is the group of isotopy classes 
of diffeomorphisms of a closed oriented surface $\Sigma_g$ of genus $g$.
A surface diffeomorphism acts on the first homology of the surface
$H_1(\Sigma_g,\Z) \simeq \Z^{2g}$ preserving the intersection form (a 
symplectic form defined over the integers). This induces a linearization map
$\Gamma_g\to \Sp(2g,\Z)$ which turns out to be onto with kernel the
Torelli group $\Tor_g$:
\begin{equation}
\label{eq.TM}
1 \longto \Tor_g \longto \Gamma_g \longto \Sp(2g,\Z) \longto 1 \,.
\end{equation}
In other words, $\Tor_g$ is the group of surface
diffeomorphisms, up to isotopy, that act trivially on the homology of 
the surface. The lower central series of any group $G$ is a graded Lie
algebra, and the action of $G$ on itself by conjugation becomes a trivial
action on its graded Lie algebra. Since $\Tor_g$ is a normal subgroup of
$\Gamma_g$, it follows that the action of $\Gamma_g$ by conjugation on
the lower central series $\ft_g$ of $\Tor_g$ descends to an action of 
$\Sp(2g,\Z)$. Said differently, the lower central series of $\Tor_g$ is 
a symplectic module, which in fact is finitely generated at each degree. 
What's more, using Mixed Hodge Structures, R. Hain showed in \cite{Hain} that 
(for $ g \geq 6$)
$\ft_g$ is a quadratic Lie algebra given by the following presentation
\begin{equation*}
\ft_g = \mathbb L(\<1^3\>)/(R)
\end{equation*}
where $\mathbb L$ is the free Lie algebra, and 
$R + \<2^2\> + \<0\>=\Wedge^2(\<1^3\>)$. In particular, it follows
that the stable algebra $\ft \eqdef \lim_{g\to\infty}\ft_g$ exists, and 
is given by the above presentation. The purpose of our paper is to 
\begin{itemize}
\item[(P1)] compute the symplectic decomposition of the lower central 
series of the Torelli group,
\item[(P2)] 
compute (stably) the Poincar\'e polynomial of the cohomology 
of the mapping group with coefficients in a symplectic representation $V$.
\end{itemize}
Our solution to (P1) uses ideas of graph cohomology to compute the
symplectic decomposition of the quadratic dual of $\ft$. Assuming that
$\ft$ is Koszul, a symplectic decomposition of $\ft$ follows. The latter,
along with Mumford's conjecture (proven by Madsen-Weiss~\cite{MW}), 
provide a solution to (P2).

Our paper combines the work of Hain, Harer, Kawazumi, Looijenga, Morita
Madsen and Weiss, and gives explicit symplectic decomposition formulas, 
and a sample of computations. Regarding computations, irreducible finite
dimensional symplectic representations are parametrized by partitions. 
The symplectic decomposition of $\ft$ in degree $n$ involves partitions with 
at most $3n$ parts, and for $n=13$ (where the answer is given in 
\cite{Ga:data}) the number of those is $177970$.  
The article was written in 1997 and a computation of the symplectic character
of $\ft$ up to degree 10 was done in Brandeis. The computation was repeated 
in 2006 in Georgia Tech and reached degree 13. Back in 1997 when the paper 
was written, it assumed Mumford's conjecture (now proven by 
Madsen-Weiss~\cite{MW}) and the Kozsulity conjecture for the graded Lie 
algebra of the Torelli group (still open), and our results predated the 
representation-stability ideas of Church-Farb~\cite{CF}. Further work on
the Torelli homomorphism and the Goldman-Turaev Lie bialgebra can be found
in Kawazumi's survey article~\cite{Goldman-Turaev}.

\subsection{The lower central series of the Torelli group}
\label{sub.hain}

Let $\Sigma_{g,s}^r$ be an oriented surface of genus $g$ with $r$ distinct 
ordered points and $s$ boundary components, and let
$\Gamma_{g,s}^r$ (resp. Let $\Tor_{g,s}^r$) be the mapping class group 
$\pi_0(\Diff_+(\Sigma_{g,s}^r))$ (resp., the subgroup of the mapping class
group that generated by surface diffeomorphisms that act trivially
on the homology of the surafce).
It is customary to omit the indices $s,r$ when they are zero. 
Recall that rational irreducible representations of the symplectic group
are indexed by partitions. 
Let $\ft_g$ denote the graded Lie algebra whose degree $n$
part equals (rationally) to the quotient of the $n$th commutator group
of the Torelli group $\Tor_g$ modulo the $(n+1)$th commutator group.
The exact sequence \eqref{eq.TM} 
implies that $\ft_g$ is a graded $R(\fsp_g)$-Lie algebra whose
universal enveloping algebra can be identified with the 
completion $\widehat{(\Q\Tor_g)}_I$ of the group ring $\Q[ \Tor_g ]$ 
with respect to its augmentation ideal $I$.
Using Mixed Hodge Structures, R. Hain showed in \cite{Hain} that 
(for $ g \geq 6$)
$\ft_g$ is a quadratic Lie algebra given by the following presentation
\begin{equation*}
\ft_g = \mathbb L(\<1^3\>)/(R)
\end{equation*}
where $\mathbb L$ is the free Lie algebra, and 
$R + \<2^2\> + \<0\>=\Wedge^2(\<1^3\>)$. In particular, it follows
that the stable algebra $\ft \eqdef \lim_{g\to\infty}\ft_g$ exists, and 
is given by the above presentation. 

\subsection{The homology of the mapping class group}
\label{sub.stability}

Harer \cite{Ha} has 
proved that the cohomology groups of the mapping class groups stabilize as
$g\to\infty$. Let $\gamma_g \in \Z[t]$ be the Poincar\'e polynomial of 
$\Gamma_g$, and let $\gamma_\infty=\lim_{g\to\infty}\gamma_g \in \Z\[t\]$. 
Mumford's conjecture, proven by Madsen-Weiss~\cite{MW},
amounts to the formula 
\begin{equation}
\label{eq.gammainf}
\gamma_\infty = \Exp \Bigl( \frac{1}{1-t^2} \Bigr) \,.
\end{equation}
Let $\HH$ be the symplectic vector space $H^1(\Sigma_g,\C)$. The symplectic
action of $\Gamma_g$ on $\HH$ induces a homomorphism from the representation
ring $R(\fsp_{2g})$ of the symplectic group to that 
$R(\Gamma_g)$ of the mapping class group. 
Using Harer's results on stability of the cohomology of the
mapping class groups, Looijenga \cite{Looijenga} has shown that the for
every representation $V \in R(\fsp_{2g})$, the cohomology groups 
$
H^\bull(\Gamma_g, V)
$
stabilize as $g\to\infty$, and will thus be denoted by
$H^\bull(\Gamma,V)$. Assuming Mumford's conjecture, Looijenga further
calculated (stably) the Poincar\'e polynomials for all symplectic 
representations $V$. 

On the other hand, using an extension due to Morita of the Johnson 
homomorphism, 
one can define classes in $H^\bull(\Gamma, \TT(\HH))$, where
$\TT(\HH)=\bigoplus_{n=0}^\infty \HH^{\o n}$. A priori, the number of these
classes may differ than the ones counted by Looijenga. It is a purpose
of the paper to compare the classes counted by Looijenga with the ones
constructed by Morita's extension and to show that they precisely agree,
assuming Mumford's conjecture. En route,
will also give a computation for the above mentioned Poincar\'e
polynomials using ideas from graph cohomology.

We now briefly recall some homomorphisms studied by Morita.
Generalizing results of Johnson \cite{Jo1,Jo2}, Morita \cite{Mo1,Mo2} defined
group homomorphisms $\rho, \rho^1$ (over $1/24 \Z$, although
we will only need their version over $\C$)
that are part of a commutative diagram
$$
\begin{CD}
\Ga_g^1 @>\rho^1>> N^1 \rtimes \Sp(\HH) \\
@VVV          @VVV  \\
\Ga_g @>\rho>> N \rtimes \Sp(\HH)
\end{CD}
$$
where $\HH=H_1(\Sigma_g,\Z)$ and $N^1, N$ are explicit torsion free
nilpotent groups, related by an exact sequence 
$1 \to \<1\>+\<1^2\> \to N^1 \to N \to 1$, and 
$\Ga_g^1 \to \Ga_g$ is induced by the
forgetful map $\Sigma_g^1 \to \Sigma_g$.
$\rho$ induces a map on cohomology $\rho^\ast:
H^\bull(N \rtimes \Sp(\HH), \TT(\HH)) \to
H^\bull(\Gamma,\TT(\HH))$.
Since $\Sp(\HH)$ is reductive, the
Lyndon-Hochshild-Serre spectral sequence \cite{HS}
(whose $E^2_{p,q}$-term equals to $H^p(\Sp(\HH), H^q(N,\TT(\HH)))$
and vanishes for $p > 0$) collapses, thus implying that
$H^\bull(N \rtimes \Sp(\HH), \TT(\HH))\cong H^\bull(N,\TT(\HH))^{\Sp}.$

Since $N$ is a nilpotent Lie group satisfying $[N,[N,N]]=0$, 
$N/[N,N]=\<1^3\>$, $[N,N]=\<2^2\>$, and the natural map $\La^2(N)\to[N,N]$
is onto, it follows that
its cohomology $H^\bull(N,\TT(\HH))$
can be identified with $\AA \otimes \TT(\HH)$, where 
\begin{equation}
\label{eq.Adef}
\AA=\Wedge^\bull (\<1^3\>)/(\<2^2\>)
\end{equation}
denotes the $R(\fsp)$-quadratic algebra 
with respect to the decomposition
$$
\Wedge^2(\<1^3\>) = \<1^6\> \oplus \<1^4\> \oplus
\<1^2\> \oplus \<0\> \oplus \<2^2,1^2\> \oplus \<2^2\>.
$$
The above discussion implies the existence of an algebra map
$$
\mu: \AA \to H^\bull(\Gamma,\TT(\HH)),
$$
where the product on the right hand side is induced by the shuffle
product on $\TT(\HH)$. 

Similarly, we obtain that the cohomology $H^\bull(N^1 \rtimes \Sp(\HH),
\TT(\HH))$ equals
to $(\AA^1 \otimes \TT(\HH))^{\Sp}$, where 
$$
\AA^1=\Wedge^\bull (\<1^3\>+ \<1\>)/(\<2^2\>+\<1^2\>)
$$
denotes the  $R(\fsp)$-quadratic algebra with respect to the decomposition 
$$
\Wedge^2(\<1^3\>+\<1\>)=
\<1^6\> \oplus 2 \<1^4\> \oplus
3\<1^2\> \oplus 2 \<0\> \oplus \<2^2,1^2\> \oplus \<2^2\> \oplus \<2,1^2\>,
$$
where $\<1^2\> \to 3 \<1^2\>$ is nonzero in all three factors. 
 This induces a map $\AA^1 \to \AA$  as well as 
$\mu^1: \AA^1 \to H^\bull(\Gamma^1,\TT(\HH))$.  

\subsection{Statement of our results}
\label{sub.results}

Our first theorem uses ideas from graph cohomology and its interpretation
as symplectic representation theory~\cite{Ko,GN}.

\begin{theorem}
\label{thm.A}
The character $\cch_t(\AA)$ of $\AA$ and $\cch_t(\AA^1)$ of $\AA^1$
in $\Lambda\[t\]$ is given by
\begin{align*}
 \cch_t(\AA) &
=\tomega \exp(-\DD') \Exp \left( \ch_t(\v) \right) \\ 
\cch_t(\AA^1) &
=\tomega \exp(-\DD') \Exp \left( h_1 t + \ch_t(\v) \right),
\end{align*}
where 
$$
\ch_t(\v)= -t h_1 -h_2+ \sum_{2k-2+n \geq 0}t^{2k-2+n} 
h_n = 
\frac{1}{t^2}\left( \frac{\Exp(th_1)}{1-t^2} -1 - (t+t^3)h_1 - t^2h_2 \right),
$$
$\tomega: \Lambda\to\Lambda$ is the involution on the ring of symmetric
functions defined by
$\tomega(p_n)=-p_n$ and
$$
\DD'= \sum_{n=1}^\infty \left( \frac{n}{2} \frac{\p^2}{\p p_n^2} +
\frac{\p}{\p p_{2n}} \right) .
$$ 
\end{theorem}

\begin{theorem}
\label{thm.2}
The maps $\mu$ and $\mu^1$ 
are algebra isomorphisms that are part of a commutative
diagram
$$\begin{CD}
\AA @>\mu>> H^\bull(\Gamma,\TT(\HH)) \\
@VVV     @VVV \\
\AA^1 @>\mu^1>> H^\bull(\Gamma^1,\TT(\HH)).
\end{CD}
$$
\end{theorem}

The rext result assumes a Koszulity conjecture for the quadratic algebra
$\AA$. For a detailed discussion on Koszul duality, see \cite{PP}.

\begin{theorem}
\label{thm.torelli}
Assuming that $\mathsf{U}(\ft)$ (or equivalently, $\AA$) is Koszul, 
the $\fsp$-character $\cch_t(\ft)$ of $\ft$ in $\Lambda\[t\]$ is given by
$$
\cch_t(\ft)=t^2-\Log(\cch_{-t}(\AA)). 
$$
\end{theorem}

\begin{remark}
In degree $n \leq 5$, Theorem \ref{thm.torelli} agrees with independent 
calculations of the character of $\ft$ by S. Morita, 
giving evidence for the Koszulity of $\AA$.
\end{remark}

\subsection{Acknowledgment}
The article was written in 1997 while both authors were in Boston.
The first author wishes to thank I. Gessel for computing guidance in Brandeis
in the fall of 1997. The authors were supported in part by the National Science 
Foundation.


\section{Symmetric functions}
\label{sec.symmetric}

\subsection{The ring of symmetric functions}
\label{sub.symmetric}
 
In this section, we recall some results on symmetric functions and
representations of the symmetric, general linear and symplectic groups
which we need later. For the proofs of these
results, we refer to Macdonald \cite{Macdonald}.

The ring of symmetric functions is the inverse limit
$$
\Lambda = \varprojlim \Z[x_1,\dots,x_k]^{\SS_k} .
$$
It is a polynomial ring in the complete symmetric functions
$$
h_n = \sum_{i_1\le\dots\le i_n} x_{i_1}\dots x_{i_n} .
$$
We may also introduce the elementary symmetric functions
$$
e_n = \sum_{i_1<\dots<i_n} x_{i_1}\dots x_{i_n} ,
$$
with generating function
$$
E(t) = \sum_{n=0}^\infty t^n e_n = \prod_i (1+tx_i) = H(-t)^{-1} .
$$
We see that $\Lambda$ is also the polynomial ring in the $e_n$. Let
$\om:\Lambda\to\Lambda$ be the involution of $\Lambda$ which maps $h_n$ to
$e_n$. The power sums (also known as Newton polynomials)
$$
p_n = \sum_i x_i^n
$$
form a set of generators of the polynomial ring $\Lambda_\Q=\Lambda\o\Q$.
This is shown by means of the elementary formula
\begin{equation} \label{P-H}
P(t) = t \frac{d}{dt} \log H(t) ,
\end{equation}
where
$$
H(t) = \sum_{n=0}^\infty t^n h_n = \prod_i (1-tx_i)^{-1}
\quad\text{and}\quad
P(t) = \sum_{n=0}^\infty t^n p_n = \sum_i (1-tx_i)^{-1} .
$$
Written out explicitly, we obtain Newton's formula relating the two sets of
generators:
$$
nh_n = p_n + h_1p_{n-1} + \dots + h_{n-1}p_1 .
$$
The involution $\om$ acts on the power sums by the formula $\om p_n =
(-1)^{n-1}p_n$.

A partition $\lambda$ is a decreasing sequence
$(\lambda_1\ge\dots\ge\lambda_\ell)$ of positive integers; we write
$|\lambda|=\lambda_1+\dots+\lambda_\ell$, and denote the length $\ell$ of
$\lambda$ by $\ell(\lambda)$. Associate to a partition $\lambda$ the
monomial $h_\lambda = h_{\lambda_1} \dots h_{\lambda_\ell}$ in the complete
symmetric functions $h_n$. The subgroup $\Lambda_n\subset\Lambda$ of
symmetric functions homogeneous of degree $n$ is a free abelian group of
rank $p(n)$, with bases $\{h_\lambda\mid |\lambda|=n\}$ and
$\{e_\lambda\mid |\lambda|=n\}$.

We may also associate to a partition $\lambda$ the monomial $p_\lambda =
p_{\lambda_1} \dots p_{\lambda_\ell}$ in the power sums $p_n$; the vector
space $\Lambda_n\o\Q$ has basis $\{p_\lambda \mid |\lambda|=n\}$. Inverting
\eqref{P-H}, we obtain the formula
\begin{equation} \label{H-P}
H(t) = \exp \Bigl( \sum_{n=1}^\infty t^n \frac{p_n}{n} \Bigr) = \sum_\lambda
t^{|\lambda|} \frac{p_\lambda}{z(\lambda)} .
\end{equation}
The integers $z(\lambda)$ arise in many places: for example, the conjugacy
class $\CO_\lambda$ of $\SS_n$ labelled by the partition $\lambda$ of $n$,
consisting of those permutations whose cycles have length
$(\lambda_1,\dots,\lambda_\ell)$, has $n!/z(\lambda)$ elements.

\subsection{Schur functions} 

There is an identification of $\Lambda$ with the ring of characters
$$
\Lambda = R(\gl) = \varprojlim R(\gl_r) ,
$$
where $\gl_r$ is the Lie algebra of $\GL(r,\C)$, obtained by mapping
$e_\lambda$ to the representation
$$
\Wedge^{\lambda_1}(\C^\infty) \o \dots \o \Wedge^{\lambda_\ell}(\C^\infty)
= \varprojlim \Wedge^{\lambda_1}(\C^r) \o \dots \o
\Wedge^{\lambda_\ell}(\C^r) .
$$
Thus, $\Lambda$ has a basis consisting of the characters of the irreducible
polynomial
representations of $\gl$. These characters, given by the Jacoby-Trudy
formula $s_\lambda = \det\bigl(h_{\lambda_i-i+j}\bigr)_{1\le i,j\le\ell}$,
are known as the Schur functions.

\subsection{Symplectic Schur functions} 

There is also an identification of $\Lambda$ with the ring of characters
$$
\Lambda = R(\fsp) = \varprojlim R(\fsp_{2g}) ,
$$
where $\fsp_{2g}$ is the Lie algebra of $\Sp(2g,\C)$, obtained by mapping
$e_\lambda$ to the representation
$$
\Wedge^{\lambda_1}(\C^\infty) \o \dots \o \Wedge^{\lambda_\ell}(\C^\infty)
= \varprojlim \Wedge^{\lambda_1}(\C^{2g}) \o \dots \o
\Wedge^{\lambda_\ell}(\C^{2g}) .
$$
Thus, $\Lambda$ has a basis consisting of the characters of the irreducible
representations $\<\lambda\>$ of $\fsp$. These characters, given by the
symplectic Jacoby-Trudy formula 
\begin{equation}
\label{eq.spjacobi}
s_{\<\lambda\>} = \frac12 \det\bigl( h_{\lambda_i-i+j} +
h_{\lambda_i-i-j+2} \bigr)_{1\le i\le\ell} ,
\end{equation}
are known as the symplectic Schur functions. 

\subsection{The Frobenius characteristic} 

If $V$ is a finite-dimensional
$\SS_n$-module $V$ with character $\chi_V:\SS_n\to\Z$, its characteristic
is the symmetric function
$$
\ch_n(V) = \sum_{|\lambda|=n} \chi_V(\CO_\lambda)
\frac{p_\lambda}{z(\lambda)} .
$$
Although it appears from its definition that $\ch_n(V)$ is in
$\Lambda\o\Q$, it may be proved that it actually lies in $\Lambda$.
\begin{proposition}
The characteristic map induces an isomorphism of abelian groups
$$
\ch_n : R(\SS_n) \to \Lambda_n .
$$
\end{proposition}

Let $\SS$ be the groupoid formed by taking the union of the symmetric
groups $\SS_n$, $n\ge0$. An $\SS$-module is a functor $n\to\v(n)$ from
$\SS$ to the category of $\N$-graded vector spaces, finite dimensional
in each degree. If $\v$ is an $\SS$-module
we define its characteristic to be the sum
$$
\ch_t(\v) = \sum_{n=0}^\infty \sum_i (-t)^i \ch_n(\v_i(n)) ,
$$
where $\v_i(n)$ is the degree $i$ component of $\v(n)$. In our paper,
we will consider examples of $\SS$-modules that either come from 
\begin{itemize}
\item
geometry, such as $\m$ and $\CC_g^n$ given in 
Section \ref{sec.mcgroup}, or from
\item
graph complexes, such as $\AA$, $\b_{\tp}$, $\b$ and $\v$ given in
Section \ref{sec.graph}, or from
\item
topology, such as $\ft_g$ and $\ft$ given in Section \ref{sub.hain}.
\end{itemize}

In summary, we have four different realizations of the same object
$\Lambda$: $\Lambda$ itself, $R(\gl)$, $R(\fsp)$ and
$R(\SS)=\bigoplus_{n=0}^\infty R(\SS_n)$. The induced isomorphisms between
$R(\SS)$, $R(\gl)$ and $R(\fsp)$ are induced by the Schur functor, which
associates to an $\SS_n$-module $V$ the $\gl$-module
$$
\varprojlim \bigl( V \o (\C^r)^{\o n} \bigr)^{\SS_n}
$$
and the $\fsp$-module
$$
\varprojlim \bigl( V \o (\C^{2g})^{\o n} \bigr)^{\SS_n} .
$$

If $\lambda$ is a partition of $n$, we denote by $V_{\lambda}$ the irreducible 
representation of $\SS_n$ with characteristic the Schur
function $\ch_n(V_{\lambda})=s_{\lambda}$. For $\lambda$ a partition of $n$,
and for a complex vector space $W$, the vector space $\bigl( V_\lambda \o
W^{\o n} \bigr)^{\SS_n}$ is denoted $\Schur^\lambda(W)$. We have the
isomorphism of $\SS_n$-modules
\begin{equation} \label{Weyl}
W^{\o n} \cong \bigoplus_{|\lambda|=n} V_\lambda \o \Schur^\lambda(W) .
\end{equation}
On the other hand, if $W$ is a complex symplectic vector space of dimension
$2g$ and $\lambda$ is a partition with at most $g$ parts, the
$\fsp_{2g}$-module $\Schur^\lambda(W)$ will not in general be irreducible;
its submodule of highest weight, which is unique, is denoted
$\Schur^{\<\lambda\>}(W)$; we denote the $\fsp_{2g}$-module
$\Schur^{\<\lambda\>}(\C^{2g})$ by $\<\lambda\>$, generalizing the case
where $g\to\infty$. The analogue of \eqref{Weyl} for symplectic vector
spaces is
\begin{equation} \label{Weyl:symplectic}
W^{\o n} \cong \bigoplus_{k=0}^{[\frac{n}{2}]} \bigoplus_{|\lambda|=n-2k}
\Ind_{\SS_k\wr\SS_2\times\SS_{n-2k}}^{\SS_n} \bigl( \1_-(2)^{\o k} \o
V_\lambda \bigr) \o \Schur^{\<\lambda\>}(W) ,
\end{equation}
where $\1_-(2)$ is the alternating character of $\SS_2$ and 
$\wr$ denotes the wreath product. 

The products on the rings $R(\gl)$ and $R(\fsp)$ correspond to the product
on the ring $\Lambda$. However, the product on $R(\SS)$ is perhaps not as
familiar: it is given by the formula
$$
(\v \times \w)(n) = \bigoplus_{k=0}^n \Ind^{\SS_n}_{\SS_k\times\SS_{n-k}}
\bigl( \v(k) \o \w(n-k) \bigr) .
$$

\subsection{The Hall inner product} 

There is a non-degenerate integral
bilinear form on $\Lambda$, denoted $\<f,g\>$, for which the Schur
functions $s_\lambda$ form an orthonormal basis. Cauchy's formula (I.4.2 of
Macdonald \cite{Macdonald})
\begin{equation} \label{Cauchy}
H(t)(...,x_iy_j,...) = \prod_{i,j} (1-tx_iy_j)^{-1} = \sum_{|\lambda|=n}
s_\lambda(x) \o s_\lambda(y) = \exp\Bigl( \sum_{k=1}^\infty \frac{p_k(x) \o
p_k(y)}{k} \Bigr) .
\end{equation}
implies that
\begin{equation} \label{Hall}
\< p_\lambda , p_\mu \> = z(\lambda) \delta_{\lambda\mu} .
\end{equation}

The adjoint of multiplication by $f\in\Lambda$ with respect to the inner
product on $\Lambda$ is denoted $f^\perp$; in particular, $\<f,g\>=(f^\perp
g)(0)$. When written in terms of the power-sums, the operator $f^\perp$
becomes a differential operator (Ex.\ 5.3 of Macdonald \cite{Macdonald}):
it follows from \eqref{Hall} that
\begin{equation} \label{D(f)}
f(p_1,p_2,\dots)^\perp = f\bigl( \tfrac{\p}{\p p_1} , 2\tfrac{\p}{\p p_2} ,
3\tfrac{\p}{\p p_3} , \dots \bigr) .
\end{equation}

If $\v$ and $\w$ are $\SS$-modules,
$$
\ch_t(\v)^\perp \ch_t(\w) = \sum_{k=0}^\infty \sum_{n=k}^\infty \ch_{n-k}
\Bigl( \Hom_{\SS_k} \Bigl( \v(k) , \Res^{\SS_n}_{\SS_k\times\SS_{n-k}}
\w(n) \Bigr) \Bigr) .
$$
Taking the dimension of $\ch_t(\v)^\perp \ch_t(\w)(0)$, we obtain the
formula
$$
\<\ch_t(\v),\ch_t(\w)\> = \sum_{n=0}^\infty \sum_i (-t)^i \dim
\bigl(\Hom_{\SS_n}(\v(n),\w(n))\bigr)_i .
$$

\subsection{Plethysm}

Aside from the product, there is another associative operation $f\circ g$
on $\Lambda$, called plethysm, which is characterized by the formulas
\begin{enumerate}
\item $(f_1+f_2)\circ g=f_1\circ g+f_2\circ g$;
\item $(f_1f_2)\circ g=(f_1\circ g)(f_2\circ g)$;
\item if $f=f(p_1,p_2,\dots)$, then $p_n\circ f=f(p_n,p_{2n},\dots)$.
\end{enumerate}

The corresponding operation on $\SS$-modules is called composition.
\begin{proposition}
If $\v$ and $\w$ are $\SS$-modules, let
$$
(\v\circ\w)(n) = \bigoplus_{k=0}^\infty \Bigl( \v(k) \o
\bigoplus_{f:[n]\to[k]}  \w(f^{-1}(1)) \o \dots \o \w(f^{-1}(k))
\Bigr)_{\SS_k} ,
$$
where $[n]=\{1,\dots,n\}$ and $[k]=\{1,\dots,k\}$. We have
$\ch_t(\v\circ\w) = \ch_t(\v)\circ\ch_t(\w)$.
\end{proposition}
\begin{proof}
When $\v$ and $\w$ are ungraded, this is proved in Macdonald
\cite{Macdonald}. In the general case, the proof depends on an analysis of
the interplay between the minus signs in the Euler characteristic and the
action of symmetric groups on tensor powers of graded vector spaces.
\end{proof}

The operation 
\begin{equation}
\label{eq.Exp}
\Exp(f)=\sum_{n=0}^\infty h_n\circ f
\end{equation}
plays the role of the
exponential in $\Lambda$. It takes values in the completion of $\Lambda$
with respect to the ideal $\ker(\eps)$ where $\eps:\Lambda\to\Z$ is the
homomorphism $\eps(f)=f(0)$ which sends $h_n$ to $0$ for $n>0$. We will
discuss this completion at greater length in the next section. 
The operation $\Exp$ extends to the $\lambda$-ring $\Lambda\[t\]$
by $p_n \circ t = t^n$, and satisfies the property
\begin{equation}
\label{eq.expab}
\Exp(f+g)=\Exp(f)\Exp(g)
\end{equation}

\begin{proposition}
We have the formula $\Exp(e_2)^\perp=\exp(\DD)$, where
$$
\DD = \sum_{n=1}^\infty \left( \frac{n}{2} \frac{\p^2}{\p p_n^2} -
\frac{\p}{\p p_{2n}} \right) .
$$
\end{proposition}
\begin{proof}
By \eqref{D(f)}, it suffices to substitute $n\p/\p p_n$ and $2n\p/\p
p_{2n}$ for $p_n$ and $p_{2n}$ on the right-hand side of
$$
\Exp(e_2) = \exp\left(\sum_{n=1}^\infty \frac{p_n}{n} \right) \circ \left(
\half (p_1^2-p_2) \right) = \exp\left(\sum_{n=1}^\infty \frac{1}{2n}
(p_n^2-p_{2n}) \right) .
\qed$$
\def\qed{}\end{proof}

Using this heat-kernel, we can relate the bases $s_\lambda$ and
$s_{\<\lambda\>}$ of $\Lambda$ corresponding respectively to the
irreducible representations in $R(\gl)$ and $R(\fsp)$. The following formula
is a consequence of \eqref{Weyl:symplectic}:
\begin{equation}
\label{eq.stosp}
s_{\<\lambda\>} = \exp(-\DD) s_\lambda .
\end{equation}


\section{Stable cohomology of mapping class groups with 
symplectic coefficients}
\label{sec.mcgroup}

In this section, we briefly present Looijenga's calculation of the stable 
cohomology of the mapping class groups with arbitrary symplectic 
coefficients.
Define a linear map $\Int:R(\fsp)\to\Z\[t\]$
by the formula
$$
\Int\bigl[s_{\<\lambda\>}\bigr] = \lim_{g\to\infty} \sum_i (-t)^i
H^i(\Gamma_g,\Schur^{\<\lambda\>}(\HH)) .
$$
In this section, we calculate the map $\Int$ 
(and its natural extension $\Int: R(\fsp)\[t\]\ohat\Lambda \to \Lambda\[t\]$
on the complete $\lambda$-ring) explicitly, by means of a
calculation in $R(\fsp)\[t\]\ohat\Lambda$; our
basic ingredients are the results of Looijenga \cite{Looijenga}.

Let $\CM_g$ be the moduli stack of genus $g$ curves (that is, the homotopy
quotient of Teichm\"uller space by $\Gamma_g$), and for $n>0$, let
$\CM_g^n$ be the moduli stack of genus $g$ curves with a configuration of
$n$ ordered distinct points. Let $\m_g$ be the $\SS$-module
$\m_g(n)=H^\bull(\CM_g^n,\C)$, and $\m$ its stable version.

\begin{theorem}
$$
 \ch_t(\m) = \gamma_\infty \Exp \Bigl( \frac{h_1}{1-t^2}
\Bigr)
$$
\end{theorem}
\begin{proof}
In \cite[Prop.2.2]{Looijenga} Looijenga proves that in degree less than $g$,
$$
H^\bull(\CM_g^n,\C) \cong H^\bull(\CM_g,\C) \o \C[u_1,\dots,u_n] ,
$$
where the classes $u_i$ of degree $2$ are the Chern classes of the line 
bundles whose
fibre at $[\Sigma,z_1,\dots,z_n]$ is $T_{z_i}^*\Sigma$. If $\1(n)$ is the
trivial $\SS_n$-module and $\C[u]$ is the $\SS$-module such that
$\C[u](1)=\C[u]$ and $\C[u](n)=0$ for $n\ne1$, we have 
$$
\1(n) \circ \C[u] \cong \C[u_1,\dots,u_n] .
$$
Taking the characteristic of both sides, we see that
$$
\ch_t(\C[u_1,\dots,u_n]) = h_n \circ \frac{h_1}{1-t^2} .
$$
Summing over $n$ and using Equation \eqref{eq.Exp} concludes the proof.
\end{proof}
Let $\CC_g^n$ be the $n$th fibred product of the universal curve
$\CM_g^1\to\CM_g$. There is an open embedding
$\CM_g^n\hookrightarrow\CC_g^n$, whose image is the complement of a divisor
with normal crossings. Let $\c_g$ be the $\SS$-module
$\c_g(n)=H^\bull(\CC_g^n,\C)$, and $\c$ its stable version.

\begin{theorem}
$$
\ch_t(\c) = \Int\bigl[\Exp\bigl( (1-t\HH+t^2)h_1
\bigr)\bigr]
$$
\end{theorem}
\begin{proof}
Since $\pi$ is a smooth projective morphism, the Leray-Serre spectral
sequence for the projection $\pi:\CC_g^n\to\CM_g$ collapses at its
$E_2$-term $H^p(\CM_g,R^q\pi_*\C)$. Observe that
$$
R^q\pi_*\C \cong \bigoplus_{\substack{j+k+\ell=n\\k+2\ell=q}}
\Ind^{\SS_n}_{\SS_j\times\SS_k\times\SS_\ell} \bigl( \C^{\o j} \o
\HH[-1]^{\o k} \o \C[-2]^{\o\ell} \bigr) ,
$$
and hence that
$$
H^i(\CC_g^n,\C) \cong \bigoplus_{p=0}^i \bigoplus_{\substack{j+k+\ell=n \\
k+2\ell=q}} \Ind^{\SS_n}_{\SS_j\times\SS_k\times\SS_\ell} \bigl( \C^{\o j}
\o H^p(\CM_g,\HH[-1]^{\o k}) \o \C[-2]^{\o\ell} \bigr) .
$$
By the main theorem of Looijenga \cite{Looijenga}, the right-hand side
stabilizes for $p<g/2+k$. Taking characteristics, we obtain
\begin{align*}
\sum_{i=0}^\infty (-t)^i \ch_n(H^i(\CC_g^n,\C)) &= \sum_{j+k+\ell=n}
\sigma_j(p_1) \cdot \Int\bigl[\sigma_k(-t\HH h_1)\bigr] \cdot
\sigma_\ell(t^2h_1) \\ &= \Int\bigl[\sigma_k((1-t\HH+t^2)h_1)\bigr] .
\qed\end{align*} \def\qed{}
\end{proof}

\begin{theorem}
$\ch_t(\c) = \ch_t(\m) \circ t^{-2} \bigl(\Exp(t^2h_1)-1\bigr)$
\end{theorem}
\begin{proof}
Fill in ... Easy, given results in Looijenga.
\end{proof}

\begin{corollary}
$\displaystyle  \ch_t(\c) = \gamma_\infty \Exp
\frac{1}{t^2} \left( \frac{\Exp(t^2h_1)-1}{1-t^2} \right)$
\end{corollary}

Combining all of these results, we obtain a formula for the operation
$\Int$.
\begin{theorem}
$$
\Int\bigl[\Exp\bigl( \HH h_1 \bigr)\bigr] = \frac{\gamma_\infty}{\Exp\bigl(
(1-t^2)^{-1} \bigr)} \Exp \frac{1}{t^2} \left( \frac{\Exp(-th_1)}{1-t^2} -
1 + (t+t^3)e_1 \right)
$$
\end{theorem}
\begin{proof}
The above theorems imply that
$$
\Int\bigl[\Exp\bigl( (1-t\HH+t^2) h_1 \bigr)\bigr] = \gamma_\infty \Exp
\frac{1}{t^2} \left( \frac{\Exp(t^2h_1)-1}{1-t^2} \right) .
$$
Multiplying both sides by $\Exp\bigl(-(1+t^2)h_1\bigr)$ gives
$$
\Int\bigl[\Exp\bigl( - t \HH h_1 \bigr)\bigr] = \gamma_\infty \Exp
\frac{1}{t^2} \left( \frac{\Exp(t^2h_1)-1}{1-t^2} - t^2(1+t^2)h_1 \right) .
$$
Replacing $h_1$ by $-h_1/t$, we obtain
$$
\Int\bigl[\Exp\bigl( \HH h_1 \bigr)\bigr] = \gamma_\infty \Exp
\frac{1}{t^2} \left( \frac{\Exp(-th_1)-1}{1-t^2} + t(1+t^2)h_1 \right) ,
$$
and the result follows.
\end{proof}

\begin{corollary}
\label{cor.slambda}
\begin{align*}
\sum_\lambda \Int\bigl[s_{\<\lambda\>}\bigr] s_\lambda &=
\Int\bigl[\exp(-\DD)\Exp\bigl( \HH h_1 \bigr)\bigr] \\ &=
\frac{\gamma_\infty}{\Exp\bigl( (1-t^2)^{-1} \bigr)} \Exp \frac{1}{t^2}
\left( \frac{\Exp(-th_1)}{1-t^2} - 1 + (t+t^3)e_1 - t^2e_2 \right) \\&=
\frac{\gamma_\infty}{\Exp\bigl( (1-t^2)^{-1} \bigr)} \tomega
\Exp \frac{1}{t^2}
\left( \frac{\Exp(th_1)}{1-t^2} - 1 - (t+t^3)h_1 - t^2h_2 \right) \\ &=
\frac{\gamma_\infty}{\Exp\bigl( (1-t^2)^{-1} \bigr)} \tomega
\Exp\left(\ch_t \v\right).
\end{align*}
\end{corollary}


\section{Graph cohomology and the $\mathfrak{sp}$-character 
of $\AA$ }
\label{sec.graph}

In this section, we will calculate the $\fsp$-character of $\AA$ using
the ideas of graph cohomology, which interprets classical invariant theory
of Lie groups in terms of graphs. For similar applications of graph cohomology
see~\cite{Ko,GN}. 

Let $\cch_t(\AA)=\sum_{k=0}^\infty (-t)^k \AA_k $ denote the
character of $\AA$ in $\Lambda\[t\]$.
Let $\b_{\tp}$ be the $\SS$-module given by
\begin{equation}
\label{eq.ab1}
\b_{\tp}(n)=(\AA^\bullet \o \TT^n_{\tp}(\HH))^{\fsp},
\end{equation}
where $\TT_{\tp}(\HH)$ is the quotient
of the algebra $\TT(\HH)$ modulo the two sided ideal generated by
$\omega_{\text{sympl}} \in \TT^2(\HH)$, where $\omega_{\text{sympl}}$
is the symplectic form on $\HH$. 

\begin{proposition}
\label{prop.b1}
$$ 
\ch_t(\b_{\tp})=\Exp \left(\ch_t(\v) \right),
$$
where $\ch_t(\v)$ is as in the statement of theorem \ref{thm.A}.
\end{proposition}

\begin{corollary}
\label{cor.b1}
\begin{equation}
\label{eq.ab2}
\cch_t(\AA)=\tomega\exp(-\DD')\ch_t(\b_{\tp}),
\end{equation}
which implies part of theorem \ref{thm.A}.
\end{corollary}

\begin{proof}
It follows from the fact that
$$\TT_{\tp}(\HH)=\bigoplus_{\lambda} \Schur^{\<\lambda\>}(\HH)\o
\Schur^{\lambda}(\HH),
$$
as well as equation \eqref{eq.stosp} 
and $\exp(-\DD) \tomega =\tomega \exp(-\DD')$.
\end{proof}

\begin{corollary}
\label{cor.b2}
$$ \frac{\gamma_\infty}{\Exp((1-t^2)^{-1})} \cch_t(\AA)
=\sum_{\lambda}\Int \bigl[s_{\<\lambda\>}\bigr] s_{\<\lambda\>}
$$
\end{corollary}

\begin{proof}
Compare corollaries \ref{cor.slambda} and \ref{cor.b1}.
\end{proof}

\begin{proof}(of proposition \ref{prop.b1})
Recall, from \cite{KontsevichManin}
and \cite[Section 2.5]{GetzlerKapranov}, that a graph $G$ is a finite set
$\Flag(G)$ (whose elements are called flags) together with an involution
$\sigma$ and a partition $\lambda$. (By a partition of a set, we mean
a disjoint decomposition into several unordered, possibly empty, subsets
called blocks).

The vertices of $G$ are the blocks of the partition $\lambda$, and the set
of them is denoted by $\VVert(G)$. The subset of $\Flag(G)$ corresponding to
a vertex $v$ is denoted by $\Leg(v)$. Its cardinality is called the valence
of $v$, and denoted by $n(v)$. The degree of a graph with $k$ trivalent
vertices and $n$ legs equals to $2k+n$. This agrees with the grading of
stable graphs given in \cite{GetzlerKapranov}, if the genus label of every
trivalent vertex is zero.

The edges of $G$ are the pairs of flags forming a two-cycle of $\sigma$,
and the set of them is denoted by $\Edge(G)$. The legs of $G$ are the fixed
points of $\sigma$, and the set of them is denoted by $\Leg(G)$.

Following \cite{Weyl}, we  review a description of
the $\fsp$-invariant algebra $\TT(\HH)^{\fsp}$ in terms of an algebra of graphs.
A chord diagram of degree $m$ is a graph $(\sigma,\lambda^1_{2m})$ 
where $\sigma$ is an involution of the set
$[2m]=\{1,\dots,2m\}$ without fixed points and 
$\lambda^1_{2m}$ is the partition of the set
$[2m]$ in all one-element subsets. 
We consider the two flags in each chord as ordered. Given a chord diagram
of degree $2m$, one can associate to it a $\fsp$-invariant tensor in
$\TT^{2m}(\HH)^{\fsp}$ by placing a copy of the symplectic form on each
edge. Upon changing the order of the two flags in a chord, the associated
invariant tensor changes sign.
Thus, we get a map
\begin{equation}
\label{eq.chord}
\cc\dd  \to \TT(\HH)^{\fsp}
\end{equation} 
where $\cc\dd$ is the quotient of the vector space 
$\C\<\text{chord diagrams}\>$ over $\C$ spanned
by all chord diagrams modulo the relation $O_1$ shown in  figure 
\ref{f.relations}.
The above map is a stable algebra isomorphism,
where the product of chord diagrams is the disjoint union.

The above map can describe the $\fsp$-invariant part of
several quotients of the tensor algebra $\TT(\HH)$.
For every quotient $\aa$ of the tensor algebra $\TT(\HH)$ which we  
consider below, there is an onto map from the algebra
$\cc\dd$ to a
combinatorial algebra $\cc\aa$, together with a commutative diagram
$$
\begin{CD}
\cc\dd @>>> \TT(\HH)^{\fsp} \\
@VVV @VVV \\
\cc\aa @>>> \aa^{\fsp}
\end{CD}
$$
such that the map $\cc\aa\to\aa^{\fsp}$ is a stable isomorphism of graded 
algebras (which multiplies degrees by $2$).
We now discuss some quotients of the tensor algebra.

The natural projection $\TT^3(\HH)\to\Wedge^3(\HH)$ induces a map
$\TT(\TT^3(\HH))\to\TT(\Wedge^3(\HH))$. The corresponding quotient of
$\cc\dd$ is the algebra $\C\<\text{T-graphs-no legs}'\>$
of trivalent graphs (without legs,
equipped with an ordering of the set of their vertices,
a cyclic ordering of the three flags around each vertex, as well as
an orientation on each edge),
modulo the relations $(O_1,O_2)$, where $O_2$ is shown in figure 
\ref{f.relations}.
The map $\cc\dd\to\C\<\text{T-graphs-no legs}'\>/(O_1,O_2)$ 
sends a chord diagram
$(\sigma,\lambda^1_{6m})$ of degree $6m$ to the trivalent graph
$(\sigma,\lambda^3_{6m})$ of degree $6m$, where
$\lambda^3_{6m}=\{\{1,2,3\},\{4,5,6,\},\dots,\{6m-2,6m-1,6m\}\}$,
thus inducing a stable isomorphism of algebras
$$
\C\<\text{T-graphs-no legs}'\>/(O_1,O_2) \to \TT(\Wedge^3(\HH))^{\fsp}.
$$

The projection
$\Wedge^3(\HH)\to\<1^3\>$, 
induces a map $\TT(\Wedge^3(\HH))\to\TT(\<1^3\>)$, where throughout the
text, all projections of $\fsp$-modules will be well defined
up to a nonzero scalar.
The corresponding quotient of $\cc\dd$ is the algebra
$\C\<\text{T-graphs-no legs}'\>/(O_1,O_2,\lloop)$, where
$\lloop$ is the relation shown in figure \ref{f.relations},
thus inducing a  stable isomorphism of algebras
$$
\C\<\text{T-graphs-no legs}'\>(O_1,O_2,\lloop) \to \TT(\<1^3\>)^{\fsp}.
$$

Consider the projection $\TT(\<1^3\>)\to\Wedge(\<1^3\>)$.
The corresponding quotient of $\cc\dd$ is the algebra
of trivalent graphs
(without legs, equipped with a sign ordering of the set of their vertices,
a cyclic ordering of the three flags around each vertex, and an
orientation on each edge),
modulo the relations $(O_1,O_2,O_3,\lloop)$, where $O_3$ is shown in figure
\ref{f.relations}. It is easy to see 
that the above algebra is isomorphic to the
algebra $\C\<\text{T-graphs-no legs}\>$ of trivalent graphs
(without orientations or legs) modulo the relation $\lloop$, 
thus inducing a  stable isomorphism of algebras
$$
\C\<\text{T-graphs-no legs}\>/(\lloop)
\to\Wedge(\<1^3\>)^{\fsp}.
$$

\begin{figure}[htpb!]
\begin{center}
\includegraphics[height=1.8cm]{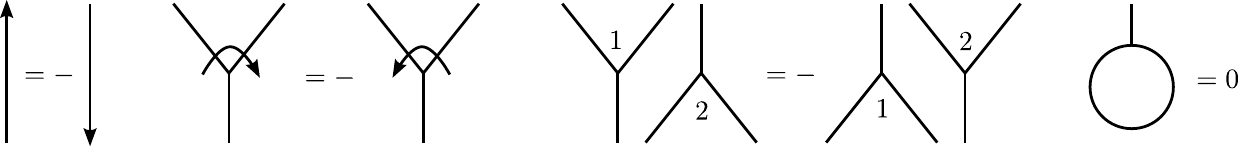}
\end{center}
\caption{The antisymmetry relations $O_1,O_2,O_3$ and the $\lloop$ 
relation. 
\label{f.relations}}
\end{figure}


Consider the projection $\Wedge(\<1^3\>)\to\AA$.
It was shown in \cite{GN} that the corresponding quotient of the algebra 
$\cc\dd$ is the algebra of trivalent graphs (without any orientations) 
modulo the relations $(\IH,\lloop)$,
where $\IH$ is shown in figure \ref{f.IH}.
 In addition, it was shown that the algebra
$\C\<\text{T-graphs}\>/(\IH,\lloop)$ is isomorphic to 
a free polynomial algebra $\Q[e_{2,0},e_{3,0},\dots]$
(where $e_{k,0}$, shown in figure \ref{f.ekn}, is of degree $2k-2$),
thus inducing a  stable isomorphism of algebras
\begin{equation}
\label{eq.Afsp}
\Q[e_{2,0},e_{3,0},\dots] \to\AA^{\fsp}.
\end{equation}

\begin{figure}[htpb!]
\begin{center}
\includegraphics[height=1.8cm]{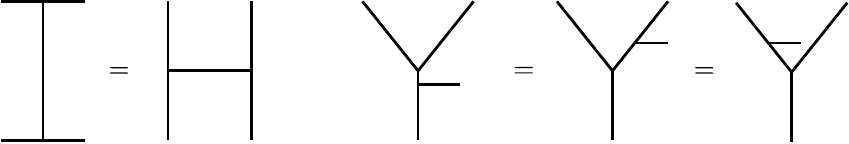}
\end{center}
\caption{On the left, the $\IH$ relation, with the understanding that the 
flags of $I$ and $H$ are not part of an edge. On the right, a consequence
of the $\IH$ relation.
\label{f.IH}}
\end{figure}


Consider the projection $\Wedge(\<1^3\>)\o\TT(\HH)\to\AA\o\TT(\HH)$.
The above discussion implies that the corresponding quotient of 
$\cc\dd$ is the algebra $\C\<\text{T-graphs}\>$
of trivalent graphs with ordered legs 
(equipped with an orientation on each edge
that connects two legs), modulo the relations $\IH,\lloop$.  
We claim that every connected  such graph 
with $2k-2+n$ trivalent vertices and $n$ legs, equals, modulo
the $\IH$ relation to the graph $e_{k,n}$ shown in figure \ref{f.ekn}.
Indeed, using the $\IH$ relation as in figure \ref{f.ekn}, we can move
edges touching a leg anywhere around the graph, thus we can
assume that there are no legs (i.e., $n=0$), in which case the result
follows from a previous discussion. Notice that $e_{k,n}$ is a trivial
representation of $\SS_n$.

\begin{figure}[htpb!]
\begin{center}
\includegraphics[height=1.8cm]{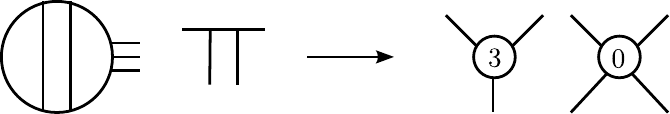}
\end{center}
\caption{The T-graph $e_{k,n}$ for $k\ge 1$ has $k-1$ vertical edges on 
a circle and $n$ horizontal legs. The T-graph $e_{0,n}$ for $n \ge 2$ is 
a tree with $n-2$ vertical legs. On the left, the graphs $e_{3,3}$ and 
$e_{0,4}$, and on the right their images. 
\label{f.ekn}}
\end{figure}


The stable isomorphism of algebras
$$
\C\<\text{T-graphs}\>/(\IH,\lloop)\to(\AA \otimes \TT(\HH))^{\fsp},
$$
implies that the $\SS$-module $\b$ defined by
$$
\b(n)=(\AA^\bullet \o \TT^n(\HH))^{\fsp}
$$
equals to $\e \circ( e_{0,2}+ \v)$, where $\e$ is the $\SS$-module
whose characteristic $\ch_t(\e)$ equals to $\sum_{h \geq 0} h_n$,
and 
$\v$ is the $\SS$-module with basis $e_{k,n}\in\v_{2k-2+n}(n)$,
$k \geq 0$, where $e_{k,n}$ spans a copy of the trivial representation
of $\SS_n$, excluding $e_{0,0}, e_{1,0}, e_{1,1}$ and $e_{0,2}$.
It follows that the characteristic
of $\v$ is given by the statement of theorem \ref{thm.A}, and that
$\ch_t(\b)=\Exp(h_2 +\ch_t(\v))$.

Finally, consider the projection $\Wedge(\<1^3\>)\o\TT(\HH)\to\AA\o
\TT_{\tp}(\HH)$.
The above discussion implies that the corresponding quotient of 
$\cc\dd$ is the quotient of the algebra $\C\<\text{T-graphs}\>$
modulo graphs some component of which contains an edge connecting two legs;
thus obtaining that $\b_{\tp}=\e \circ \v$, which concludes the proof
of proposition \ref{prop.b1}.
\end{proof}

Similarly, we have a stable isomorphism of algebras
$$
\C\<\text{T-graphs}\>/(\IH^1)\to(\AA^1 \otimes \TT(\HH))^{\fsp},
$$
(where $\IH^1$ is the relation of figure \ref{f.IH},
assuming that at most two of the four flags of the graphs $I$ and $H$ belong
to the same edge), which implies that
the $\SS$-module $\b^1$ defined by
$$
\b^1(n)=(\AA^\bullet \o \TT^n(\HH))^{\fsp}
$$
equals to $\e \circ(e_{0,2}+e_{1,1} + \v)$, and that the 
$\SS$-module $\b^1_{\tp}$ defined by
$$ \b^1_{\tp}(n)=(\AA^{1,\bullet} \o \TT^n_{\tp}(\HH))^{\fsp}
$$
equals to $\e \circ(e_{1,1} + \v)$, thus
$$
 \cch_t(\AA^1)=\tomega\exp(-\DD')\ch_t(\b^1_{\tp})=
\Exp(h_1 t + \ch_t(\v))
$$
which concludes the proof of theorem \ref{thm.A}.

\begin{remark}
\label{rem.di}
There is a curious degree preserving map from 
t-graphs $e_{k,n}$ of degree $2k-2+n > 0$ with $n$ legs
to stable graphs (in the terminology of \cite{GetzlerKapranov})
with one vertex of genus $k$ and $n$ legs, see figure
\ref{f.ekn}. There is also a similarity between the character 
$\ch_t(\v)$ of $\v$ and the Feynman transform of \cite{GetzlerKapranov}.
Notice also that the $\SS$-module $\v$ has an additional multiplication
$$\di: \v(n) \o \v(m) \to \v(n+m-2)$$
 defined by joining a leg
of $e_{k,n}$ with one of $e_{l,m}$, in other words by $e_{k,n} \di e_{l,m}
=e_{k+l,n+m-2}$. Under the isomorphism 
$ \b_{\tp} \cong \e \circ \v$, $\di$ 
corresponds to a map $\di: \b_{\tp}(n) \o \b_{\tp}(m) \to \b_{\tp}(n+m-2)$ 
(denoted by the same name) 
defined using the product in $\AA$ and the contraction map
$$\di: \TT(\HH) \o \TT(\HH) \to \TT(\HH)$$
 given by 
$$(a_1 \o \dots \o a_n) \o (b_1 \o \dots \o b_m) \to
\sum_{i,j} (a_i \cdot b_j) a_1 \o \dots \hat{a}_i \dots \o a_n \o
b_1 \o \dots \hat{b}_j \dots \o b_m,
$$  
where $a_i,b_j \in \HH$ and $\cdot: \HH \o \HH \to \C$ is given by the 
symplectic form.
\end{remark}


\section{Proof of theorem \ref{thm.2}}

Let $\n$ denote the $\SS$-module $\n(n)=H^\bull(\Ga,\TT^n(\HH))$.
Under the isomorphism $\b \cong \e \circ \v$, the discussion in the 
introduction implies the existence of a map 
$$\mu_F: \b(n) \to \n(n)$$
which is part of a commutative diagram:
$$
\begin{CD}
\b(n) \o \b(m) @>\mu_F(n) \o \mu_F(m) >> \n(n) \o \n(m) \\
@V\di VV                                  @VV \di V  \\
\b(n+m-2)      @>\mu_F(n+m-2)>>           \n(n+m-2)
\end{CD}
$$
where the vertical map on the right hand side is induced by the the cup product
follows by a contraction on $\TT(\HH)$.

Note that the left vertical map is onto. We claim that assuming Mumford's
conjecture, 
\begin{itemize}
\item
The right vertical map is onto
\item
$\mu_F(0)$ is onto (follows from Morita's papers, 
assuming Mumford's conjecture. We could give a proof here anyway.)
\item
$\mu_F(1)$ is onto.
\end{itemize}
Then, $\mu_F$ is onto and since $\ch_t(\b)=\ch_t(\n)$, $\mu_F$ is an
isomorphism, thus proving theorem \ref{thm.2}.


Let us do as an example the following special case for $R=\TT^3(\HH)$.
Assuming Mumford's conjecture
it follows from corollary \ref{cor.slambda} that $\text{dim}H^1(\Ga,
\TT^3(\HH))=1$, and that $\text{dim}H^1(\Ga^1,
\TT^3(\HH))=4$. 
With the notation of the introduction, after projecting 
$N \to N/[N,N]=\<1^3\>$, 
$\rho$ can be thought of as a crossed homomorphism
$\Ga_g \to \<1^3\> \hookrightarrow R$, 
in other words as a 1-cocycle of $\Ga_g$ with values
in $R$; let $[\rho] \in H^1(\Ga_g, R)$ denote its cohomology class.
Then $\mu_F(1)(e_{0,3})=[\rho]$, and 
Morita proves that $[\rho]$ is nonzero, by an explicit cocycle computation.
Similarly, $\mu^1_F(1)(e_{0,3})=[\rho^1] \in H^1(\Ga^1_g, R)$
is nonzero too.

$$
\begin{CD}
(A_1 \o R)^{\fsp} @>>> H^1(\Ga,R) \\
@VVV                      @VVV       \\
(A^1_1 \o R)^{\fsp} @>>> H^1(\Ga^1,R)
\end{CD}
\text{ where }
\begin{CD}
e_{0,3}           @>>> [\rho] \\
@VVV                           @VVV  \\
3 e_{1,1}e_{0,2} + e_{0,3} @>>> [\rho^1] \\
\end{CD}
$$
where $2 e_{1,1} e_{0,2}$ correspond to the three labelings of the legs
of the graph $e_{1,1} e_{0,2}$, and these maps are up to scalar multiples
that are missing.

Note that Morita's map $k_0$ \cite{KM} is simply
$\mu_F(2)(e_{0,2}) \in H^1(\CM_g^2,\HH)$, and that the contraction formula
that they have must be included in our commutative diagram above.





\section{The symplectic character of the Torelli Lie algebra}
\label{sec.koszul}

In this section we give a proof of Theorem \ref{thm.torelli} using
properties of quadratic algebras and Koszul duality, explained in detail
in \cite{PP}.
Recall from Section \ref{sub.hain} the quadratic presentation of the
$R(\fsp)$-Lie algebra $\ft_g$. Hain also established
a central extension of graded Lie algebras
$$
1 \to \Q(2) \to \ft_g \to \fu_g \to 1,
$$
where $\Q(2)$ is a one dimensional abelian subalgebra of $\ft_g$
in degree $2$. The characteristic of the Lie algebra $\ft_g$ and 
its universal enveloping algebra $\mathsf{U}(\ft_g)$ are related
by
$$
\cch_t(\ft_g)=\Log(\cch_t(\mathsf{U}(\ft_g))=t^2 + 
\Log(\cch_t(\mathsf{U}(\fu_g))
$$
On the other hand, the quadratic dual $\mathsf{U}(\fu_g)^!$ of 
$\mathsf{U}(\ft_g)$ equals, by definition, 
to $\AA$; see \cite[Chpt.1.2]{PP}. It follows that that $\AA$ is Koszul 
if and only if $\fu$ is Koszul if and only if $\ft$ is Koszul. 
Since Koszul dual algebras $(B,B^!)$ statisfy $\cch_t(B)\cch_{-t}(B^!)=1$ 
(see \cite[Cor.2.2]{PP}), this concludes the proof of Theorem 
\ref{thm.torelli}.


\section{Computations}
\label{sec.compute}


Theorem \ref{thm.A} implies that the character of $\AA_n$ is a linear
combination of symplectic Schur functions $s_{\<\lambda\>}$ for
$|\lambda| \le 3n$, and can thus be calculated by truncating $\Exp$ 
in degrees at most $3n$. Note that partitions with $|\lambda|=3n$
appear in $\AA_n$. 

Assuming that $\AA$ is Koszul, Theorem \ref{thm.torelli}
implies that the character of $\ft_n$ is a linear
combination of symplectic Schur functions $s_{\<\lambda\>}$ for
$|\lambda| \le n+2$, and can be calculated by truncating $\Exp$ 
in degrees at most $3n$. 

A computation of the character of $\AA$ and (assuming
Koszulity) of $\ft$ was obtained up to degree $12$, 
using J. Stembridge's \cite{Stembridge} symmetric function package {\tt SF} 
for {\tt maple}. Below, we present 
the results for $\AA_n$ for $n=1,\dots,4$ and for $\ft_n$ 
for $n=1,\dots,8$. A table for $n \leq 13$ is available in~\cite{Ga:data}.

To explain our computations, we add a few comments on some of our equations.

Equation~\eqref{eq.stosp} states the following:
if $f=\sum_\lambda a_\lambda s_\lambda$, then 
$\exp(-\DD) f = \sum_\lambda a_\lambda s_{\<\lambda\>}$.

Using $\Exp(f+g)=\Exp(f)\Exp(g)$ and $h_n(t)=t^n$ and 
$\Exp(t)=\sum_{n=0}^\infty h_n(t)=(1-t)^{-1}$, we see that the right 
hand side of Equation~\eqref{eq.gammainf} is given by:
$$
\Exp \Bigl( \frac{1}{1-t^2} \Bigr) = 
\Exp \Bigl( \sum_{n=0}^\infty t^{2n} \Bigr)=\prod_{n=0}^\infty
\Exp ( t^{2n} )=\frac{1}{\prod_{n=1}^\infty (1-t^{2n})} 
$$

Corollary \ref{cor.b1} states that if
$\tomega \, \ch_t(\b_{\tp}) =\sum_{\lambda} a_{\lambda} s_{\lambda}$, 
then
$\cch_t(\AA) = \sum_{\lambda} a_{\lambda} s_{\<\lambda \>}$.

The next lemma is an effective way to compute the character $\Exp(\ch_t(\v))$
using the {\tt SF} package.

\begin{lemma}
\label{lem.compute}
We have:
\begin{equation}
\label{eq.chtv}
\Exp(\ch_t(\v)) = \sum_{\lambda} \gamma_{\lambda}
\end{equation}
where
\begin{equation}
\label{eq.gammal}
\gamma_{(1^{l_1}2^{l_2}\dots)} =\prod_n h_{l_n} \circ \frac{t^{c_n} h_n}{1-t^2}
=t^{\delta(\lambda)} \prod_n h_{l_n} \circ \frac{h_n}{1-t^2}
\end{equation}
and $c_n=n-2$ (resp., $0$, $3$, $2$) when $n \geq 3$ 
(resp., $n=0$, $n=1$, $n=2$) and $\d(1^{l_1}2^{l_2}\dots)=\sum_n c_n l_n$.
\end{lemma}
Note that to calculate the $\Lambda_n\[t\]$ part of $\Exp(\ch_t(\v))$, 
we can truncate Equation \eqref{eq.chtv} to partitions $\lambda$ with 
$|\lambda| \leq 3n$. Indeed, the worst case occures when $\lambda=(3^n)$
when $|\lambda|=3n$ and $\d(\lambda)=n$. Compare also with 
\cite[Eqn.(1)]{Looijenga}.

\begin{proof}
Using Equations \eqref{eq.Exp} and \eqref{eq.expab}, 
we have
\begin{eqnarray*}
\Exp(\ch_t(\v)) &=& \Exp\Bigl( \frac{1}{1-t^2} + \frac{t^3 h_1}{1-t^2} +
\frac{t^2 h_2}{1-t^2} + \sum_{n=3}^\infty \frac{t^{n-2} h_n}{1-t^2} 
\Bigr) \\
&=& 
\prod_{n=0}^\infty \Exp\Bigl(\frac{t^{c_n} h_n}{1-t^2} \Bigr) \\
&=&
\prod_{n=0}^\infty \sum_{l_n=0}^\infty h_{l_n} \circ \frac{t^{c_n} h_n}{1-t^2} 
\end{eqnarray*}
Encoding partitions $\lambda=(1^{l_1}2^{l_2}\dots)$ by their number of parts,
the first  result follows. Since $h_{l_n}$ is homogeneous of degree $l_n$,
we have $h_{l_n} \circ t^{c_n} h_n=t^{l_n c_n}  h_{l_n} \circ h_n$. The result
follows.
\end{proof}


\begin{center}
\begin{tabular}{|c|l|}
\hline
$n$ & $\AA_n$ \\ \hline
$1$ & $\<1^3\>$ \\
$2$ & $\<0\> + \<1^4\> + \<1^2\> + \<1^6\> + \<2^2, 1^2\>$ \\ 
$3$ & $\<1^9\> + \<2^2, 1^5\> + \<2^3, 1^3\> + \<3^2, 1^3\> + \<3, 2^3\> 
  + 2\,\<1^5\> + \<2, 1^3\> + \<2^2, 1\> + \<1^7\> + \<2, 1^5\> 
  $ \\ & $+ \<2^2, 1^3\> + \<2^3, 1\> + 2\,\<1^3\> + \<1\>$ \\ 
$4$ & $2\,\<0\> + 4\,\<2^2, 1^4\> + 
\<3, 2^3, 1\> + 4\,\<1^6\> + 4\,\<1^4\> + 2\,\<1^2\> 
  + 3\,\<2^2, 1^2\> + \<2^2\> + \<3^4\> + 2\,\<2^3, 1^2\> 
  $ \\ & $+ 3\,\<1^8\> 
  + 2\,\<2^4\> + \<3^2, 1^2\> + \<2^2, 1^8\> + \<3, 2^2, 1\> + 2\,\<2, 1^6\> 
  + \<1^{12}\> + \<3, 2^3, 1^3\> + \<2^3, 1^6\> 
  $ \\ & $+ \<3^2, 1^6\> + \<2^4, 1^4\> + 2\,\<2^4, 1^2\> + 2\,\<2^3, 1^4\> 
  + \<3, 2, 1^5\> + \<3, 2^2, 1^3\> + \<3^2, 1^4\> + \<3^2, 2^2\> 
  $ \\ & $+ \<3^2, 2, 1^2\> 
  + \<2, 1^2\> + \<2, 1^8\> + 2\,\<2^2, 1^6\> 
  + \<1^{10}\> 
  + \<3, 2, 1^3\> + \<2^3\> + 2\,\<2, 1^4\> + \<3^2, 2, 1^4\> 
  $ \\ & $+ \<2^6\> 
  + \<3^2, 2^2, 1^2\> + \<4^2, 1^4\> + \<4, 3, 2^2, 1\> $\\
\hline
\end{tabular}
\end{center}


\begin{center}
\begin{tabular}{|c|l|}
\hline
$n$ & Degree $n$ part of $\ft$ \\ \hline
$1$ & $\<1^3\>$ \\ [0.20cm] 
$2$ & $\<0\> + \<2^2\>$ \\ [0.20cm] 
$3$ & $\<3,1^2\>$ \\ [0.20cm] 
$4$ & $\<2\> + \<3,1\> + \<4,2\>  + \<2^3\> + \<3,1^3\>$ \\ [0.20cm] 
$5$ & $\<2, 1\> + \<1^3\> + \<4, 1\> + \<3, 2\> + \<3, 1^2\> + \<2^2, 1\>
+ \<2, 1^3\> 
+ \<5, 1^2\> + \<4, 2, 1\> + \<3^2, 1\>$ \\ [0.20cm] & $+ \<3, 2, 1^2\> + \<2^2, 1^3\>
$ \\ [0.20cm]
$6$ & $ \<2\> + \<0\> + \<2^2, 1^4\> + \<2^4\> + \<3, 2, 1^3\>
     + \<3, 2^2, 1\>  +  2 \<3^2, 1^2\> + \<4, 1^4\> + \<4, 2, 1^2\>
     + 2 \<4, 2^2\>$ \\ & $ + \<4, 3, 1\> + \<4^2\> +  \<5, 1^3\> 
     + \<5, 2, 1\>
     + \<6, 2\> + \<1^6\> + \<2, 1^4\> +  3 \<2^2, 1^2\>
     + \<2^3\> + 2 \<3, 1^3\> $ \\ & $+ 4 \<3, 2, 1\> + 2 \<3^2\>
     + 2 \<4, 1^2\> +  2 \<4, 2\> + 2 \<5, 1\> + 2 \<1^4\>
     + 2 \<2, 1^2\> + 5 \<2^2\> + 2 \<3, 1\> + 2 \<4\> 
     $ \\ & $+ 2 \<1^2\> $ \\ [0.20cm]
$7$ & $ \<7, 1^2\> + 6\<4, 1^3\> + 10\<4, 2, 1\> + 4\<4, 3\>
      + 4\<5, 1^2\> + 4\<5, 2\> + 2\<6, 1\> + 13\<3, 1^2\> +  7\<4, 1\>
      $ \\ & $+ \<5\> + 4\<3\> + 2\<3^3\> + \<4, 1^5\> + 2\<4, 3, 2\>
      + 3\<5, 2, 1, 1\> + \<5, 2^2\> + 3\<5, 3, 1\> + \<6, 1^3\>
      $ \\ & $+ \<6, 2, 1\> + \<6, 3\> + 3\<3, 2^2, 1^2\> + \<3, 1^6\>
      + 2\<3^2, 2, 1\> + 3\<4, 3, 1^2\> + 3\<4, 2^2, 1\>
      + 3\<4, 2, 1^3\> $ \\ & $+ \<4^2, 1\> + \<3, 2, 1^4\>
      + \<2^3, 1^3\> + \<3^2, 1^3\> + \<3, 2^3\> + 7\<2, 1\>
      + 5\<1^3\> + 2\<2, 1^5\> + 4\<2^2, 1^3\>
      $ \\ & $+ 4\<2^3, 1\> + 4\<3, 1^4\> + 10\<3, 2, 1^2\>
      + 6\<3, 2^2\> + 6\<3^2, 1\> + \<1\> + 8\<2^2, 1\> 
      + 8\<2, 1^3\> + \<1^5\> $ \\ & $+ 8\<3, 2\> $  \\ [0.20cm]
$8$ & $  15 \<2\> + 5 \<2^2, 1^4\> + 3 \<2^4\>
      + 20 \<3, 2, 1^3\> + 23 \<3, 2^2, 1\> + 15 \<3^2, 1^2\>
      + 11 \<4, 1^4\> +  29 \<4, 2, 1^2\> $ \\ & $ + 17 \<4, 2^2\>
      + 22 \<4, 3, 1\> + 3 \<4^2\> + 10 \<5, 1^3\> + 20 \<5, 2, 1\>
      + 6 \<6, 2\> + \<1^6\> + 16 \<2, 1^4\>
      $ \\ & $+ 20 \<2^2, 1^2\> + 20 \<2^3\> + 32 \<3, 1^3\> + 45 \<3, 2, 1\>
      + 9 \<3^2\> + 25 \<4, 1^2\> + 31 \<4, 2\> +   12 \<5, 1\>
      $ \\ & $+ 6 \<1^4\> + 29 \<2, 1^2\> + 16 \<2^2\> + 31 \<3, 1\> + 8 \<4\>
      + 5 \<1^2\> + 3 \<7, 1\> + 5 \<6\> +  8 \<5, 3, 1^2\> 
      $ \\ & $+ 4 \<5, 3, 2\>
      + 3 \<5, 4, 1\> + 2 \<6, 1^4\> + 4 \<6, 2^2\> + 3 \<6, 3, 1\>
      + 2 \<6, 4\> + \<7, 1^3\> + 2 \<7, 2, 1\> 
      $ \\ & $+ \<8, 2\> + 10 \<5, 3\>
      + \<4, 3^2\> + \<5, 1^5\> + 5 \<5, 2, 1^3\>
      + 5 \<5, 2^2, 1\> + 3 \<6, 2, 1^2\> + 6 \<6, 1^2\> 
      $ \\ & $+ 3 \<3^3, 1\>
      + \<3^2, 2^2\> + 5 \<4, 2^2, 1^2\> + 2 \<2, 1^6\>
      + 5 \<4, 2, 1^4\> + \<2^3, 1^4\> + 9 \<4, 3, 2, 1\>
      + 5 \<4, 2^3\> $ \\ & $+ 6 \<4, 3, 1^3\> + \<4^2, 1^2\> + 5 \<4^2, 2\>
      + 14 \<3^2, 2\> + 9 \<2^3, 1^2\> + 8 \<3, 1^5\>
      + 4 \<3, 2^2, 1^3\> + \<3, 1^7\>
      $ \\ & $+ 2 \<3, 2, 1^5\> + 2 \<2^5\> + 6 \<3^2, 2, 1^2\>
      + \<3^2, 1^4\> + 3 \<3, 2^3, 1\> $
\\ \hline 
\end{tabular}
\end{center}


\appendix

\section{$\lambda$-rings}

Using the ring $\Lambda$, it is possible to give a very direct definition
of Gro\-then\-dieck's $\lambda$- and special $\lambda$-rings. We follow
more recent usage in referring to special $\lambda$-rings as
$\lambda$-rings, and to $\lambda$-rings as pre-$\lambda$-rings.

\subsection{Pre-$\lambda$-rings} A pre-$\lambda$-ring is a commutative ring
$R$, together with a morphism of commutative rings $\sigma_t:R\to R\[t\]$
such that $\sigma_t(a)=1+ta+O(t^2)$. Expanding $\sigma_t$ in a power series
$$
\sigma_t(a) = \sum_{n=0}^\infty t^n \sigma_n(a) ,
$$
we obtain endomorphisms $\sigma_n$ of $R$ such that $\sigma_0(a)=1$,
$\sigma_1(a)=a$, and
$$
\sigma_n(a+b) = \sum_{k=0}^n \sigma_{n-k}(a) \sigma_k(b) .
$$
There are also operations $\lambda_k(a)=(-1)^k\sigma_k(-a)$, with
generating function
\begin{equation} \label{invert}
\lambda_t(a) = \sum_{n=0}^\infty t^n \lambda_n(a) = \sigma_{-t}(a)^{-1} .
\end{equation}
The $\lambda$-operations are polynomials in the $\sigma$-operations with
integral coefficients, and vice versa. In this paper, we take the
$\sigma$-operations to be more fundamental; nevertheless, following custom,
the structure they define is called a pre-$\lambda$-ring.

Given a pre-$\lambda$-ring $R$, there is a bilinear map $\Lambda\o R\to R$,
which we denote $f\circ a$, defined by the formula
$$
(h_{n_1}\dots h_{n_k})\circ a = \sigma_{n_1}(a)\dots\sigma_{n_k}(a) .
$$
The image of the power sum $p_n$ under this map is the operation on $R$
known as the Adams operation $\psi_n$.  We denote the operation
corresponding to the Schur function $s_\lambda$ by $\sigma_\lambda$. Note
that \eqref{H-P} implies the relation
$$
\sigma_t(a) = \exp \Bigl( \sum_{n=1}^\infty \frac{t^n\psi_n(a)}{n} \Bigr) ,
$$
from which the following result is immediate.
\begin{proposition}
If $R$ and $S$ are pre-$\lambda$-rings, their tensor product $R\o S$ is a
pre-$\lambda$-ring, with $\sigma$-operations
$$
\sigma_n(a\o b) = \sum_{|\lambda|=n} \sigma_\lambda(a) \o \sigma_\lambda(b)
,
$$
and Adams operations $\psi_n(a\o b) = \psi_n(a) \o \psi_n(b)$.
\end{proposition}

For example, $\sigma_2(a\o b) = \sigma_2(a)\o\sigma_2(b) +
\lambda_2(a)\o\lambda_2(b)$.

\subsection{$\lambda$-rings} 

The polynomial ring $\Z[x]$ is a
pre-$\lambda$-ring, with $\sigma$-operations characterized by the formula
$\sigma_n(x^i)=x^{ni}$. Taking tensor powers of this pre-$\lambda$-ring
with itself, we see that the polynomial ring $\Z[x_1,\dots,x_k]$ is a
pre-$\lambda$-ring. The $\lambda$-operations on this ring are equivariant
with respect to the permutation action of the symmetric group $\SS_k$ on
the generators, hence the ring of symmetric functions
$\Z[x_1,\dots,x_k]^{\SS_k}$ is a pre-$\lambda$-ring. Taking the limit
$k\to\infty$, we obtain a pre-$\lambda$-ring structure on $\Lambda$.

\begin{definition}
A $\lambda$-ring is pre-$\lambda$-ring such that if $f,g\in\Lambda$ and
$x\in R$,
\begin{equation}
\label{lambda-ring}
f\circ(g\circ x)=(f\circ g)\circ x .
\end{equation}
\end{definition}

By definition, the pre-$\lambda$-ring $\Lambda$ is a $\lambda$-ring; in
particular, the operation $f\circ g$, called plethysm, is associative.

The following result (see Knutson \cite{Knutson}) is the chief result in
the theory of $\lambda$-rings.
\begin{theorem}
\label{universal}
$\Lambda$ is the universal $\lambda$-ring on a single generator $h_1$.
\end{theorem}

This theorem makes it straightforward to verify identities in
$\lambda$-rings: it suffices to verify them in $\Lambda$. As an
application, we have the following corollary.

\begin{corollary}
The tensor product of two $\lambda$-rings is a $\lambda$-ring.
\end{corollary}

\begin{proof}
We need only verify this for $R=\Lambda$. A torsion-free pre-$\lambda$-ring
whose Adams operations are ring homomorphisms which satisfy
$\psi_m(\psi_n(a))=\psi_{mn}(a)$ is a $\lambda$-ring.  It is easy to verify
these conditions for $\Lambda\o\Lambda$, since $\psi_n(a\o b) = \psi_n(a)
\o \psi_n(b)$.
\end{proof}

In the definition of a $\lambda$-ring, it is usual to adjoin the axiom
$$
\sigma_n(xy) = \sum_{|\lambda|=n} \sigma_\lambda(a) \o \sigma_\lambda(y) .
$$
However, this formula follows from our definition of a $\lambda$-ring: by
universality, it suffices to check it for $R=\Lambda\o\Lambda$, $x=h_1\o1$
and $y=1\o h_1$, for which it is evident.

\subsection{Complete $\lambda$-rings} 

A filtered $\lambda$-ring $R$ is a
$\lambda$-ring with decreasing filtration
$$
R = F^0R \supset F^1R \supset \dots ,
$$
such that
\begin{enumerate}
\item $\bigcap_k F^kR = 0$ (the filtration is discrete);
\item $F^mR\*F^nR\subset F^{m+n}R$ (the filtration is compatible with the
product);
\item $\sigma_m(F^nR)\subset F^{mn}R$ (the filtration is compatible with
the $\lambda$-ring structure).
\end{enumerate}
The completion of a filtered $\lambda$-ring is again a $\lambda$-ring;
define a complete $\lambda$-ring to be a $\lambda$-ring equal to its
completion. For example, the universal $\lambda$-ring $\Lambda$ is filtered
by the subspaces $F^n\Lambda$ of polynomials vanishing to order $n-1$, and
its completion is the $\lambda$-ring of symmetric power series, whose
underlying ring is the power series ring $\Z\[h_1,h_2,h_3,\dots\]$.

The tensor product of two filtered $\lambda$-rings is again a filtered
$\lambda$-ring, when furnished with the filtration
$$
F^n(R\o S) = \sum_{k=0}^n F^{n-k}R \o F^kS .
$$
If $R$ and $S$ are filtered $\lambda$-rings, denote by $R\ohat S$ the
completion of $R\o S$.

If $R$ is a complete $\lambda$-ring, the operation
$$
\Exp(a) = \sum_{n=0}^\infty \sigma_n(a) : F_1R \to 1+F_1R
$$
is an analogue of exponentiation, in the sense that $\Exp(f+g) =
\Exp(f)\Exp(g)$. Its logarithm is given by a formula of Cadogan
\cite{Cadogan}.

\begin{proposition}
\label{Cadogan}
On a complete filtered $\lambda$-ring $R$, the operation
$\Exp:F_1R\to1+F_1R$ has inverse
$$
\Log(1+a) = \sum_{n=1}^\infty \frac{\mu(n)}{n} \log(1+\psi_n(a)) .
$$
\end{proposition}

\begin{proof}
Expanding $\Log(1+a)$, we obtain
$$
\Log(1+a) = - \sum_{n=1}^\infty \frac{1}{n} \sum_{d|n} \mu(d)
\psi_d(-a)^{n/d} = \sum_{n=1}^\infty \Log_n(a) .
$$
Let $\chi_n$ be the character of the cyclic group $C_n$ equalling $e^{2\pi
i/n}$ on the generator of $C_n$. The characteristic of the $\SS_n$-module
$\Ind^{\SS_n}_{C_n}\chi_n$ equals
$$
\frac{1}{n} \sum_{k=0}^{n-1} e^{2\pi ik/n} p_{(k,n)}^{n/(k,n)}
= \frac{1}{n} \sum_{d|n} \mu(d) p_d^{n/d} ,
$$
while the characteristic of the $\SS_n$-module
$\Ind^{\SS_n}_{C_n}\chi_n\o\eps_n$, where $\eps_n$ is the sign
representation of $\SS_n$, equals
$$
\frac{1}{n} \sum_{d|n} \mu(d) \bigl( (-1)^{d-1}p_d \bigr)^{n/d}
= \frac{(-1)^n}{n} \sum_{d|n} \mu(d) (-p_d)^{n/d} .
$$
It follows that $(-1)^{n-1}\Log_n$ is the operation associated to the
$\SS_n$-module $\Ind^{\SS_n}_{C_n}\chi_n\o\eps_n$, and hence defines a map
from $F_1R$ to $F_nR$.

To prove that $\Log$ is the inverse of $\Exp$, it suffices to check this
for $R=\Lambda$ and $x=h_1$. We must prove that
$$
\Exp\left( \sum_{n=1}^\infty \frac{\mu(n)}{n} \log(1+p_n) \right) = 1+h_1 .
$$
The logarithm of the expression on the left-hand side equals
$$
\exp \Bigl( \sum_{k=1}^\infty \frac{p_k}{k} \Bigr) \circ
\Bigl( \sum_{n=1}^\infty \frac{\mu(n)}{n} \log(1+p_n) \Bigr)
= \sum_{n=1}^\infty \sum_{d|n} \mu(d) \frac{\log(1+p_n)}{n}
= \log(1+p_1) ,
$$
and the formula follows.
\end{proof}

\bibliographystyle{hamsalpha}
\bibliography{biblio}

\providecommand{\bysame}{\leavevmode\hbox to3em{\hrulefill}\thinspace}
\providecommand{\href}[2]{#2}
\providecommand{\eprint}{\begingroup \urlstyle{rm}\Url}
\begin{thebibliography}{Wey39}

\bibitem[Cad71]{Cadogan}
Charles Cadogan, \emph{The {M}\"obius function and connected graphs}, J.
  Combinatorial Theory Ser. B \textbf{11} (1971), 193--200.

\bibitem[CF13]{CF}
Thomas Church and Benson Farb, \emph{Representation theory and homological
  stability}, Adv. Math. \textbf{245} (2013), 250--314.

\bibitem[Gar]{Ga:data}
Stavros Garoufalidis,
  \url{http://www.math.gatech.edu/~stavros/publications/mcgroup.data}.

\bibitem[GK98]{GetzlerKapranov}
Ezra Getzler and Mikhail Kapranov, \emph{Modular operads}, Compositio Math.
  \textbf{110} (1998), no.~1, 65--126.

\bibitem[GN98]{GN}
Stavros Garoufalidis and Hiroaki Nakamura, \emph{Some {$IHX$}-type relations on
  trivalent graphs and symplectic representation theory}, Math. Res. Lett.
  \textbf{5} (1998), no.~3, 391--402.

\bibitem[Hai97]{Hain}
Richard Hain, \emph{Infinitesimal presentations of the {T}orelli groups}, J.
  Amer. Math. Soc. \textbf{10} (1997), no.~3, 597--651.

\bibitem[Har85]{Ha}
John Harer, \emph{Stability of the homology of the mapping class groups of
  orientable surfaces}, Ann. of Math. (2) \textbf{121} (1985), no.~2, 215--249.

\bibitem[HS53]{HS}
Gerhard Hochschild and Jean-Pierre Serre, \emph{Cohomology of group
  extensions}, Trans. Amer. Math. Soc. \textbf{74} (1953), 110--134.

\bibitem[Joh80]{Jo1}
Dennis Johnson, \emph{An abelian quotient of the mapping class group
  {${\mathcal I}_{g}$}}, Math. Ann. \textbf{249} (1980), no.~3, 225--242.

\bibitem[Joh83]{Jo2}
\bysame, \emph{A survey of the {T}orelli group}, Low-dimensional topology
  ({S}an {F}rancisco, {C}alif., 1981), Contemp. Math., vol.~20, Amer. Math.
  Soc., Providence, RI, 1983, pp.~165--179.

\bibitem[KK16]{Goldman-Turaev}
Nariya Kawazumi and Yusuke Kuno, \emph{The {G}oldman-{T}uraev {L}ie bialgebra
  and the {J}ohnson homomorphisms}, Handbook of {T}eichm\"uller theory. {V}ol.
  {V}, IRMA Lect. Math. Theor. Phys., vol.~26, Eur. Math. Soc., Z\"urich, 2016,
  pp.~97--165.

\bibitem[KM94]{KontsevichManin}
Maxim Kontsevich and Yuri Manin, \emph{Gromov-{W}itten classes, quantum
  cohomology, and enumerative geometry}, Comm. Math. Phys. \textbf{164} (1994),
  no.~3, 525--562.

\bibitem[KM96]{KM}
Nariya Kawazumi and Shigeyuki Morita, \emph{The primary approximation to the
  cohomology of the moduli space of curves and cocycles for the stable
  characteristic classes}, Math. Res. Lett. \textbf{3} (1996), no.~5, 629--641.

\bibitem[Knu73]{Knutson}
Donald Knutson, \emph{{$\lambda $}-rings and the representation theory of the
  symmetric group}, Lecture Notes in Mathematics, Vol. 308, Springer-Verlag,
  Berlin, 1973.

\bibitem[Kon94]{Ko}
Maxim Kontsevich, \emph{Feynman diagrams and low-dimensional topology}, First
  {E}uropean {C}ongress of {M}athematics, {V}ol.\ {II} ({P}aris, 1992), Progr.
  Math., vol. 120, Birkh\"auser, Basel, 1994, pp.~97--121.

\bibitem[Loo96]{Looijenga}
Eduard Looijenga, \emph{Stable cohomology of the mapping class group with
  symplectic coefficients and of the universal {A}bel-{J}acobi map}, J.
  Algebraic Geom. \textbf{5} (1996), no.~1, 135--150.

\bibitem[Mac95]{Macdonald}
Ian Macdonald, \emph{Symmetric functions and {H}all polynomials}, second ed.,
  Oxford Mathematical Monographs, The Clarendon Press Oxford University Press,
  New York, 1995, With contributions by A. Zelevinsky, Oxford Science
  Publications.

\bibitem[Mor93]{Mo1}
Shigeyuki Morita, \emph{The extension of {J}ohnson's homomorphism from the
  {T}orelli group to the mapping class group}, Invent. Math. \textbf{111}
  (1993), no.~1, 197--224.

\bibitem[Mor96]{Mo2}
\bysame, \emph{A linear representation of the mapping class group of orientable
  surfaces and characteristic classes of surface bundles}, Topology and
  {T}eichm\"uller spaces ({K}atinkulta, 1995), World Sci. Publ., River Edge,
  NJ, 1996, pp.~159--186.

\bibitem[MW07]{MW}
Ib~Madsen and Michael Weiss, \emph{The stable moduli space of {R}iemann
  surfaces: {M}umford's conjecture}, Ann. of Math. (2) \textbf{165} (2007),
  no.~3, 843--941.

\bibitem[PP05]{PP}
Alexander Polishchuk and Leonid Positselski, \emph{Quadratic algebras},
  University Lecture Series, vol.~37, American Mathematical Society,
  Providence, RI, 2005.

\bibitem[Ste95]{Stembridge}
John Stembridge, \emph{{SF}}, 1995,
  \url{http://www.math.lsa.umich.edu/~jrs/maple.html}.

\bibitem[Wey39]{Weyl}
Hermann Weyl, \emph{The {C}lassical {G}roups. {T}heir {I}nvariants and
  {R}epresentations}, Princeton University Press, Princeton, N.J., 1939.

\end{thebibliography}
\end{document}